\documentclass[10pt, 3p, lefttitle]{elsarticle}
\journal{arXiv}

\usepackage{amssymb}
\usepackage[fleqn]{amsmath}
\usepackage{amsthm}
\usepackage{xcolor}
\usepackage{upgreek}
\usepackage{nicefrac}
\usepackage[hidelinks]{hyperref}
\usepackage{booktabs}
\usepackage{subcaption}

\newcommand{\bs}[1]{{\boldsymbol #1}}

\def\plus{\texttt{+}}

\renewcommand{\uplus}{u_h^{\plus}}
\newcommand{\zplus}{z_h^{\plus}}
\newcommand{\phiplus}{\varphi_h^{\plus}}
\newcommand{\eplus}{e_h^{\plus}}
\newcommand{\Vplus}{\mathcal{V}_h^{\plus}}
\newcommand{\eps}{\varepsilon}

\newtheorem{lemma}{Lemma}
\newtheorem{theorem}{Theorem}
\newtheorem{corollary}{Corollary}
\theoremstyle{definition}
\newtheorem{assumption}{Assumption}
\newtheorem{definition}{Definition}
\newtheorem{example}{Example}
\newtheorem*{results}{Numerical Demonstration}
\newtheorem{remark}{Remark}

\begin{document}

%%%%%%%%%%%%%%%%%%%%%%%%%%%%%%%%%%%%%%%%%%%%%%%%%%%%%%%%%%%%%%%%%%%%%%%%%%%%%%%
\begin{frontmatter}
%%%%%%%%%%%%%%%%%%%%%%%%%%%%%%%%%%%%%%%%%%%%%%%%%%%%%%%%%%%%%%%%%%%%%%%%%%%%%%%

\title{A Note on the Reliability of Goal-Oriented Error
Estimates for Galerkin Finite Element Methods with Nonlinear Functionals}
\author[sandia]{Brian N. Granzow\corref{correspondence}}
\author[sandia]{Stephen D. Bond}
\author[sandia]{D. Thomas Seidl}
\author[luh]{Bernhard Endtmayer}
\address[sandia]{Sandia National Laboratories. P.O. Box 5800. Albuquerque, NM. 87185-1321.}
\address[luh]{Leibniz University Hannover. Welfengarten 1. 30167. Hannover Germany. \vspace{-3em}}
\cortext[correspondence]{Corresponding author, bngranz@sandia.gov}
\fntext[label1]{
B.N. Granzow, D.T. Seidl, and S.D. Bond were supported by the Advanced Simulation
and Computing program at Sandia National Laboratories, a multimission laboratory
managed and operated by National Technology and Engineering Solutions of Sandia LLC,
a wholly owned subsidiary of Honeywell International Inc. for the
U.S. Department of Energy's
National Nuclear Security Administration under contract DE-NA0003525. This
article describes objective technical results and analysis. Any subjective
views or opinions that might be expressed in the article do not necessarily
represent the views of the U.S. Department of Energy or the United States
Government. }


%------------------------------------------------------------------------------
\begin{abstract}
We consider estimating the discretization error in a nonlinear functional
$J(u)$ in the setting of an abstract variational problem: find
$u \in \mathcal{V}$ such that $B(u,\varphi) = L(\varphi) \; \forall \varphi
\in \mathcal{V}$, as approximated by a Galerkin finite element method. Here,
$\mathcal{V}$ is a Hilbert space, $B(\cdot,\cdot)$ is a bilinear form, and
$L(\cdot)$ is a linear functional. We consider well-known error estimates
$\eta$ of the form $J(u) - J(u_h) \approx \eta = L(z) - B(u_h, z)$, where
$u_h$ denotes a finite element approximation to $u$, and $z$ denotes the
solution to an auxiliary adjoint variational problem. We show that there exist
nonlinear functionals for which error estimates of this form are not reliable,
even in the presence of an exact adjoint solution solution $z$. An estimate
$\eta$ is said to be reliable if there exists a constant
$C \in \mathbb{R}_{>0}$ independent of $u_h$ such that
$|J(u) - J(u_h)| \leq C|\eta|$. We present several example pairs of
bilinear forms and nonlinear functionals where reliability of $\eta$ is not
achieved.
\end{abstract}
%------------------------------------------------------------------------------

\end{frontmatter}

%%%%%%%%%%%%%%%%%%%%%%%%%%%%%%%%%%%%%%%%%%%%%%%%%%%%%%%%%%%%%%%%%%%%%%%%%%%%%%%
\section{Introduction}
\label{sec:introduction}
%%%%%%%%%%%%%%%%%%%%%%%%%%%%%%%%%%%%%%%%%%%%%%%%%%%%%%%%%%%%%%%%%%%%%%%%%%%%%%%

\noindent\emph{A posteriori} error estimation \cite{ainsworth1997posteriori,
stewart1998tutorial} provides useful tools to approximate and control
discretization errors in finite element settings. Goal-oriented error
estimation \cite{becker1996feed,becker2001optimal,oden2001goal,
oden2002estimation,fidkowski2011review,granzow2017output,endtmayer2020two,
endtmayer2024posteriori} measures the discretization error in a functional
quantity of interest, rather than in a global norm. Inherent to the process of
goal-oriented error estimation is a linearization and a discard of higher order
information. In a recent article \cite{granzow2023linearization}, we described
this discard as a \emph{linearization error} and presented several examples
where traditional adjoint-weighted residual error estimates fail to achieve
effectivity (an accurate approximation of the functional discretization error).
In this note, we provide further insight into this loss of effectivity by
showing that there exist specific choices of nonlinear functionals for which
certain adjoint-weighted residual estimates are not \emph{reliable}, the
setting for which is described in the abstract.

%%%%%%%%%%%%%%%%%%%%%%%%%%%%%%%%%%%%%%%%%%%%%%%%%%%%%%%%%%%%%%%%%%%%%%%%%%%%%%%
\section{Preliminaries}
%%%%%%%%%%%%%%%%%%%%%%%%%%%%%%%%%%%%%%%%%%%%%%%%%%%%%%%%%%%%%%%%%%%%%%%%%%%%%%%

\noindent Consider the following variational problems:
\begin{align}
&\mathrm{Find} \, u \in \mathcal{V}     \; \mathrm{s.t.} \; B(u, \varphi)     = L(\varphi)    \quad \forall \varphi \in \mathcal{V},      \label{eq:primal} \\
&\mathrm{Find} \, u_h \in \mathcal{V}_h \; \mathrm{s.t.} \; B(u_h, \varphi_h) = L(\varphi_h)  \quad \forall \varphi_h \in \mathcal{V}_h,  \label{eq:primal_fem} \\
&\mathrm{Find} \, \uplus \in \Vplus     \; \mathrm{s.t.} \; B(\uplus, \phiplus) = L(\phiplus) \quad \forall \phiplus \in \Vplus,          \label{eq:primal_fem2}
\end{align}
where $\mathcal{V}$ denotes a Hilbert space,
$B(\cdot, \cdot) : \mathcal{V} \times \mathcal{V} \to \mathbb{R}$
denotes a bilinear form, $L(\cdot) : \mathcal{V} \to \mathbb{R}$ denotes
a linear functional, and $\mathcal{V}_h$ and $\Vplus$ denote finite dimensional
subspaces of $\mathcal{V}$ such that $\mathcal{V}_h \subset \Vplus \subset
\mathcal{V}$, where $h \in \mathbb{R}_{>0}$ denotes a discretization parameter.

Let $J(u) : \mathcal{V} \to \mathbb{R}$ denote a nonlinear functional with
sufficient regularity such that its first and second Fr\'{e}chet
derivatives $J'$ and $J''$, respectively, exist. We seek to approximate the
(assumed well-posed) primal problem \eqref{eq:primal} with a Galerkin
finite element method \eqref{eq:primal_fem} and estimate the functional
discretization error $J(u) - J(u_h)$ with the introduction of an
auxiliary linear \emph{adjoint problem}. Presently, we consider the
following variational problems:
\begin{align}
&\mathrm{Find} \, z \in \mathcal{V} \; \mathrm{s.t.} \; B(\varphi, z) = J'(u_h; \varphi) \quad \forall \varphi \in \mathcal{V}, \label{eq:adjoint} \\
&\mathrm{Find} \, \zplus \in \Vplus \; \mathrm{s.t.} \; B(\phiplus, \zplus) = J'(u_h; \phiplus) \quad \forall \phiplus \in \Vplus, \label{eq:adjoint_fem}
\end{align}
where we refer to the first variational problem \eqref{eq:adjoint} as
the continuous adjoint problem and \eqref{eq:adjoint_fem} represents
a finite element approximation to the continuous adjoint problem.
It is well known that if the adjoint problem is approximated in the
space $\mathcal{V}_h$, then traditional adjoint-weighted residual error
estimates will be erroneously zero due to \emph{Galerkin orthogonality}:
$B(e, \varphi_h) = 0 \; \forall \varphi_h \in \mathcal{V}_h$, where
$e := u - u_h$ denotes the discretization error. To this end, we
consider approximating the adjoint problem in the richer finite-dimensional
space $\Vplus$.
\begin{remark}
Practically, we consider only standard piecewise polynomial finite element
spaces $\mathcal{V}_h$ and globally enriched spaces $\Vplus$ obtained
by increasing the piecewise polynomial order of the space $\mathcal{V}_h$ by
one. The following discussion should, however, apply in broader contexts,
provided $\mathcal{V}_h \subset \Vplus$.
\end{remark}

\begin{definition}
\label{def:functional_derivs}
(Functional derivatives) The first and second G\^{a}teaux derivatives of
a functional $J: \mathcal{V} \to \mathbb{R}$ are defined as
\begin{align}
J'(u;v) &:= \lim_{\eps \to 0} \eps^{-1} \left[ J(u + \eps v) - J(u) \right] = \frac{\mathrm{d}}{\mathrm{d} \eps} J(u + \eps v) \biggr|_{\eps = 0}, \\
J''(u;v,w) &:= \lim_{\eps \to 0} \eps^{-1} \left[ J'(u + \eps w; v) - J'(u;v) \right] = \frac{\mathrm{d}}{\mathrm{d} \eps} J'(u + \eps w; v) \biggr|_{\eps = 0},
\end{align}
respectively. These derivatives agree with their corresponding Fr\'{e}chet derivatives
when the functional is once and twice Fr\'{e}chet differentiable.
\end{definition}

\begin{lemma}
\label{lma:taylor}
Let $u \in \mathcal{V}$, $u_h \in \mathcal{V}_h$,  $\uplus \in \Vplus$,
$e:= u - u_h$ and $\eplus := \uplus - u_h$. Then the following Taylor expansions hold:
\begin{align}
J(u) - J(u_h)       &=  J'(u_h; e) + \int_0^1 J''(u_h+se;e,e)(1-s) \, \mathrm{d}s,                     \label{eq:J_taylor} \\
J(\uplus) - J(u_h)  &=  J'(u_h; \eplus) + \int_0^1 J''(u_h+s\eplus;\eplus,\eplus)(1-s) \, \mathrm{d}s. \label{eq:J_taylor_fem}
\end{align}
\end{lemma}

\begin{lemma}
\label{lma:error_representation}
Let $u \in \mathcal{V}$, $u_h \in \mathcal{V}_h$, and $z \in \mathcal{V}$
solve problems \eqref{eq:primal}, \eqref{eq:primal_fem}, and \eqref{eq:adjoint},
respectively, and let $e := u - u_h$. Then the following error representation holds
\begin{equation}
J(u) - J(u_h) = L(z) - B(u_h,z) + \int_0^1 J''(u_h+se;e,e)(1-s) \, \mathrm{d}s.
\end{equation}
\end{lemma}
\begin{proof}
The proof follows by sequentially considering equations \eqref{eq:J_taylor}
and \eqref{eq:adjoint}, applying bilinearity, and then using equation \eqref{eq:primal}.
\end{proof}

\begin{lemma}
\label{lma:error_representation_h}
Let $\uplus \in \Vplus$, $u_h \in \mathcal{V}_h$, and $\zplus \in \Vplus$ 
solve problems \eqref{eq:primal_fem2}, \eqref{eq:primal_fem}, and
\eqref{eq:adjoint_fem}, respectively, and let $\eplus := \uplus - u_h$.
Then the following error representation holds
\begin{equation}
J(\uplus) - J(u_h) = L(\zplus) - B(u_h,\zplus) + \int_0^1 J''(u_h+s\eplus;\eplus,\eplus)(1-s) \, \mathrm{d}s.
\end{equation}
\end{lemma}
\begin{proof}
The proof follows the same steps as Lemma \ref{lma:error_representation} using discrete
counterparts \eqref{eq:J_taylor_fem}, \eqref{eq:adjoint_fem}, and \eqref{eq:primal_fem2}.
\end{proof}

\begin{assumption}
\label{asm:saturation}
(Saturation) Let $u_h$ and $\uplus$ solve Equations \eqref{eq:primal_fem}
and \eqref{eq:primal_fem2}, respectively. Then there exists
$b_h \in (0,1)$ such that
\begin{equation}
|J(u) - J(\uplus)| \leq b_h |J(u) - J(u_h)|.
\end{equation}
\end{assumption}

\begin{remark}
We are unaware of a general way to verify the saturation assumption for
arbitrary nonlinear goal functionals. It is a common assumption
in hierarchical-based error estimation \cite{verfurth1996review} and has been
proven in specific contexts (e.g. \cite{carstensen2016justification}), but it has also been shown that the assumption
can fail \cite{bornemann1996posteriori, endtmayer2020two}.
\end{remark}

\begin{definition}
\label{def:reliable}
(Reliable) An error estimate $\eta \approx J(u) - J(u_h)$ is said
to be reliable if there exists a constant $C \in \mathbb{R}_{>0}$
independent of $u_h$ such that
\begin{equation}
|J(u) - J(u_h)| \leq C | \eta |.
\end{equation}
\end{definition}

\begin{lemma}
\label{lma:eta1_is_reliable}
Let $u_h$, $\uplus$, and $\zplus$ solve Equations \eqref{eq:primal_fem},
\eqref{eq:primal_fem2} and \eqref{eq:adjoint_fem}, respectively, and let
Assumption \ref{asm:saturation} hold. Then the estimate
$\eta_1 \approx J(u) - J(u_h)$ is reliable, with the constant
$C = 1/(1-b_h)$, where
\begin{equation}
\eta_1 := L(\zplus) - B(u_h, \zplus) + \int_0^1 J''(u_h+s\eplus;\eplus,\eplus)(1-s) \, \mathrm{d}s.
\label{eq:eta1}
\end{equation}
\end{lemma}
\begin{proof}
A generalized proof can be found in Endtmayer \emph{et al.} \cite{endtmayer2024posteriori}.
\end{proof}

%%%%%%%%%%%%%%%%%%%%%%%%%%%%%%%%%%%%%%%%%%%%%%%%%%%%%%%%%%%%%%%%%%%%%%%%%%%%%%%
\section{Results}
\label{sec:results}
%%%%%%%%%%%%%%%%%%%%%%%%%%%%%%%%%%%%%%%%%%%%%%%%%%%%%%%%%%%%%%%%%%%%%%%%%%%%%%%

Lemma \ref{lma:eta1_is_reliable}, in part, provides assurance that
$\eta_1$ will provide a useful error estimate. It is rare, however,
to account for the remainder term
$\int_0^1 J''(u_h + s\eplus;\eplus,\eplus)(1-s) \, \mathrm{d}s$ in
the performance of goal-oriented error estimation. The neglect of
this term is what we phrase a \emph{linearization error}. We now
show that neglecting this remainder term can lead to unreliable
error estimates for certain functionals within the context of the
chosen adjoint problem and the error estimates presented below
(see the discussions section for further elaboration).
In the following, we let $\Omega \subset \mathbb{R}^d$ denote an open,
bounded domain, where $d=1,2,3$ denotes a spatial dimension.

\begin{theorem}
\label{thm:result1}
Let $u$, $u_h$, and $z$ solve Equations \eqref{eq:primal},
\eqref{eq:primal_fem}, and \eqref{eq:adjoint}, respectively,
let $e := u - u_h$, and
let $J(u)$ be chosen so that $J'(u_h;e) = B(e, \psi_h)$ for some
$\psi_h \in \mathcal{V}_h$. Then the estimate
$\eta_2 \approx J(u) - J(u_h)$ is not reliable, where
\begin{equation}
\eta_2 := L(z) - B(u_h, z).
\label{eq:eta2}
\end{equation}
\end{theorem}
\begin{proof}
\begin{align*}
0 &= B(e, \psi_h),        && \text{by Galerkin orthogonality}, \\
  &= J'(u_h; e),          && \text{by assumption}, \\
  &= B(e, z),             && \text{by Equation \eqref{eq:adjoint}}, \\
  &= B(u, z) - B(u_h, z)  && \text{by bilinearity}, \\
  &= L(z) - B(u_h, z)     && \text{by Equation \eqref{eq:primal}}.
\end{align*}
Thus, $\eta_2 = 0$ and
$\nexists \, C \, \mathrm{s.t.} \, |J(u) - J(u_h)| \leq C |\eta_2|.$
\end{proof}

Stated otherwise, Theorem \ref{thm:result1} demonstrates that there exist
functionals for which the entirety of the functional discretization error
is contained within the second-order remainder term $J''$ in the expansion
\eqref{eq:J_taylor}, i.e. $J'(u_h;e) = 0$, because the term
$J'(u_h;e)$ mimics the Galerkin orthogonality principle of the original
variational problem.

\begin{theorem}
\label{thm:result2}
Let $u_h$, $\uplus$, and $\zplus$ solve Equations \eqref{eq:primal_fem},
\eqref{eq:primal_fem2}, and \eqref{eq:adjoint_fem}, respectively, let
$\eplus := \uplus - u_h$, and
let $J(u)$ be chosen so that $J'(u_h;\eplus) = B(\eplus, \psi_h)$ for
some $\psi_h \in \mathcal{V}_h$. Then the estimate
$\eta_3 \approx J(u) - J(u_h)$ is not reliable, where
\begin{equation}
\eta_3 := L(\zplus) - B(u_h, \zplus).
\end{equation}
\end{theorem}
\begin{proof}
Note a discrete form of Galerkin orthogonality,
$B(\eplus, \varphi_h) = 0 \; \forall \varphi_h \in \mathcal{V}_h$,
holds when $\mathcal{V}_h \subset \Vplus$. Starting from this point,
the proof follows similar steps as the proof for Theorem \ref{thm:result1}
using instead the discrete equations \eqref{eq:adjoint_fem} and \eqref{eq:primal_fem2}.
\end{proof}

\begin{theorem}
\label{thm:result3}
Let $u$, $u_h$, $\uplus$, $z$, and $\zplus$ solve Equations
\eqref{eq:primal}, \eqref{eq:primal_fem}, \eqref{eq:primal_fem2},
\eqref{eq:adjoint}, and \eqref{eq:adjoint_fem}, respectively,
let $B(\cdot, \cdot)$ be symmetric, let $G$ denote a once
differentiable function everywhere it is defined, and let
$J(u) = G(B(u,u))$. Then the estimates $\eta_2$ and $\eta_3$
from Theorems \ref{thm:result1} and \ref{thm:result2}, respectively,
are not reliable.
\end{theorem}
\begin{proof}
\begin{align*}
J'(u_h;e) &= \frac{\mathrm{d}}{\mathrm{d}\eps} G(B(u_h + \eps e, u_h + \eps e)) \biggr|_{\eps=0}, && \\
&= G'(B(u_h + \eps e, u_h + \eps e)) \frac{\mathrm{d}}{\mathrm{d}\eps} B(u_h + \eps e, u_h + \eps e) \biggr|_{\eps=0}, && \\
&= G'(B(u_h,u_h)) 2 B(u_h, e),      && \text{by bilinearity}, \\
&= 2 G'(B(u_h,u_h)) B(e, u_h),      && \text{by symmetry}, \\
&= B(e, \hat{C} u_h),               && \text{by linearity},
\end{align*}
where $\hat{C} := 2G'(B(u_h,u_h)) \in \mathbb{R}$ is a constant,
so that $\hat{C} u_h \in \mathcal{V}_h$. Theorem
\ref{thm:result1} can then be applied, showing that the
estimate $\eta_2$ is not reliable. Following the same steps
and replacing $e$ with $\eplus$ shows that
$J'(u_h; \eplus) = B(\eplus, \hat{C} u_h)$, where
Theorem \ref{thm:result2} can then be applied to show that
the estimate $\eta_3$ is not reliable.
\end{proof}

\begin{corollary}
\label{cor:result4}
Let $u$, $u_h$, $\uplus$, $z$, and $\zplus$ solve Equations
\eqref{eq:primal}, \eqref{eq:primal_fem}, \eqref{eq:primal_fem2},
\eqref{eq:adjoint}, and \eqref{eq:adjoint_fem}, respectively. Let
$B(\cdot,\cdot)$ be symmetric and let $J(u) := \sqrt{B(u,u)}$
be the energy induced by the bilinear form. Then the
estimates $\eta_2$ and $\eta_3$ from Theorems \ref{thm:result1} and
\ref{thm:result2}, respectively, are not reliable.
\end{corollary}
\begin{proof}
The result follows directly from Theorem \ref{thm:result3}, where
$G(\cdot) = \sqrt{\cdot} : \mathbb{R}_{>0} \to \mathbb{R}_{>0}$.
\end{proof}

\begin{example}
\label{ex:poisson}
(Poisson's equation):
Let $B(u,\varphi) := \int_{\Omega} \nabla u \cdot \nabla \varphi \, \mathrm{d}\Omega$
and $J(u) := \int_{\Omega} \nabla u \cdot \nabla u \, \mathrm{d}\Omega$.
Then we have
\begin{equation*}
J'(u_h; e) = \int_{\Omega} \nabla e \cdot (2 \nabla u_h) \, \mathrm{d} \Omega =
B(e, 2 u_h),
\end{equation*}
and since $2u_h \in \mathcal{V}_h$, Theorems \ref{thm:result1} and \ref{thm:result2}
hold.
\end{example}

\begin{example}
\label{ex:elasticity}
(Isotropic linear elasticity):
Let $B(u, \varphi) := \int_{\Omega} \sigma(u) : \nabla \varphi \, \mathrm{d} \Omega.$
and let $J(u) := \sqrt{\int_{\Omega} \sigma(u) : \epsilon(u)}$
be the so-called \emph{strain-energy density}. Here,
$u$ and $\varphi$ are vector-valued, $\epsilon$ and $\sigma$ are second-order
tensor-valued, $\epsilon(u) := \frac12(\nabla u + \nabla u^T)$,
and $\sigma(u) := \lambda \mathrm{tr}(\epsilon(u)) I + 2 \mu \epsilon(u)$,
where $I$ is the second-order identity tensor, $\lambda$ denotes Lam\'{e}'s first
parameter, and $\mu$ denotes a shear modulus. It not difficult to show that
\begin{align*}
\int_{\Omega} \sigma(u) : \nabla \varphi \, \mathrm{d} \Omega =
\int_{\Omega} \sigma(u) : \epsilon(\varphi) \, \mathrm{d} \Omega,
\end{align*}
and that $B(u,\varphi)$ is symmetric. Thus, $J(u) = \sqrt{B(u,u)}$ and
Corollary \ref{cor:result4} holds.
\end{example}

\begin{remark}
The chosen quantities of interest in the previous examples may not be of primary interest
when used in a single-goal error control context. However, in a multi-goal error
control context \cite{endtmayer2024posteriori}, it may be desirable to control
errors in energy norm-like functionals in addition to multiple other quantities
with the use of multi-goal oriented error estimation machinery. In such a scenario, the examples
demonstrate that the energy norm-like functionals may not be reliably accounted for.
\end{remark}

\begin{results}
\label{res:poisson}
To provide a concrete numerical demonstration, we consider Poisson's
equation posed on the domain $\Omega := \{\bs{x} : \bs{x} \in (-1,1) \times (-1,-1)
\setminus [-\nicefrac12,\nicefrac12] \times [-\nicefrac12,\nicefrac12] \}$ with
homogenous Dirichlet boundary conditions, the right-hand
side functional $L(\varphi) := \int_{\Omega} f \varphi \, \text{d} \Omega$, where $f = 100$,
and the functional quantity of interest described in Example \ref{ex:poisson}.
We consider solving the primal and adjoint problems with quadrilateral elements
with various combinations of Lagrange basis functions of orders $p$ and $q$,
respectively, where $p < q$, while using several mesh sizes $h$. Let
$\mathcal{E} := J(u) - J(u_h)$ denote the functional discretization error and
let $J(u) \approx 6.703016825 \times 10^{2}$ be a numerically determined
`true' value of the quantity of interest, obtained numerically using 
$1.5$ million degrees of freedom. Table \ref{tab:results1} illustrates that the
computable error estimate $\eta_3$ is unreliable for all tested $p,q$, and $h$
combinations, yielding an error estimate that is numerically zero (lower than the
CG iterative solver tolerance of $10^{-10}$), while the actual functional discretization error is not zero.
Figure \ref{fig:results2} shows the computational domain and mesh with
mesh size $h=\nicefrac{1}{16}$, the finite element solution $u_h$, and an
approximation of element-level contributions to the functional discretization
error $\mathcal{E}$.

\begin{table}[ht!]
\begin{center}
\begin{tabular}{ccccc}
\toprule
$h$ & $p$ & $q$ & $\mathcal{E}$ & $\eta_3$ \\
\midrule
$\nicefrac{1}{8}$  & $1$ & $2$ & $4.589 \times 10^{1}$ & $ 6.157 \times 10^{-11}$ \\
$\nicefrac{1}{8}$  & $1$ & $3$ & $4.589 \times 10^{1}$ & $-8.594 \times 10^{-11}$ \\
$\nicefrac{1}{8}$  & $1$ & $4$ & $4.589 \times 10^{1}$ & $-1.506 \times 10^{-11}$ \\
$\nicefrac{1}{8}$  & $2$ & $3$ & $1.692 \times 10^{0}$  & $-5.572 \times 10^{-11}$ \\
$\nicefrac{1}{8}$  & $2$ & $4$ & $1.692 \times 10^{0}$  & $-1.360 \times 10^{-11}$ \\
$\nicefrac{1}{16}$ & $1$ & $2$ & $1.273 \times 10^{1}$ & $-9.689 \times 10^{-12}$ \\
$\nicefrac{1}{16}$ & $1$ & $3$ & $1.273 \times 10^{1}$ & $ 4.688 \times 10^{-12}$ \\
$\nicefrac{1}{16}$ & $1$ & $4$ & $1.273 \times 10^{1}$ & $ 2.724 \times 10^{-11}$ \\
$\nicefrac{1}{16}$ & $2$ & $3$ & $6.517 \times 10^{-1}$ & $-2.796 \times 10^{-11}$ \\
$\nicefrac{1}{16}$ & $2$ & $4$ & $6.517 \times 10^{-1}$ & $-1.821 \times 10^{-11}$ \\
\bottomrule
\end{tabular}
\end{center}
\caption{The error estimate $\eta_3$ compared to the `true' error $\mathcal{E}$ for
the problem described in the numerical demonstration
when the primal and adjoint problems are solved with polynomial
orders $p$ and $q$, respectively, with varying mesh size $h$.}
\label{tab:results1}
\end{table}

\begin{figure}[ht!]
\centering
\begin{subfigure}{0.33\textwidth}
\centering
\includegraphics[width=.99\linewidth]{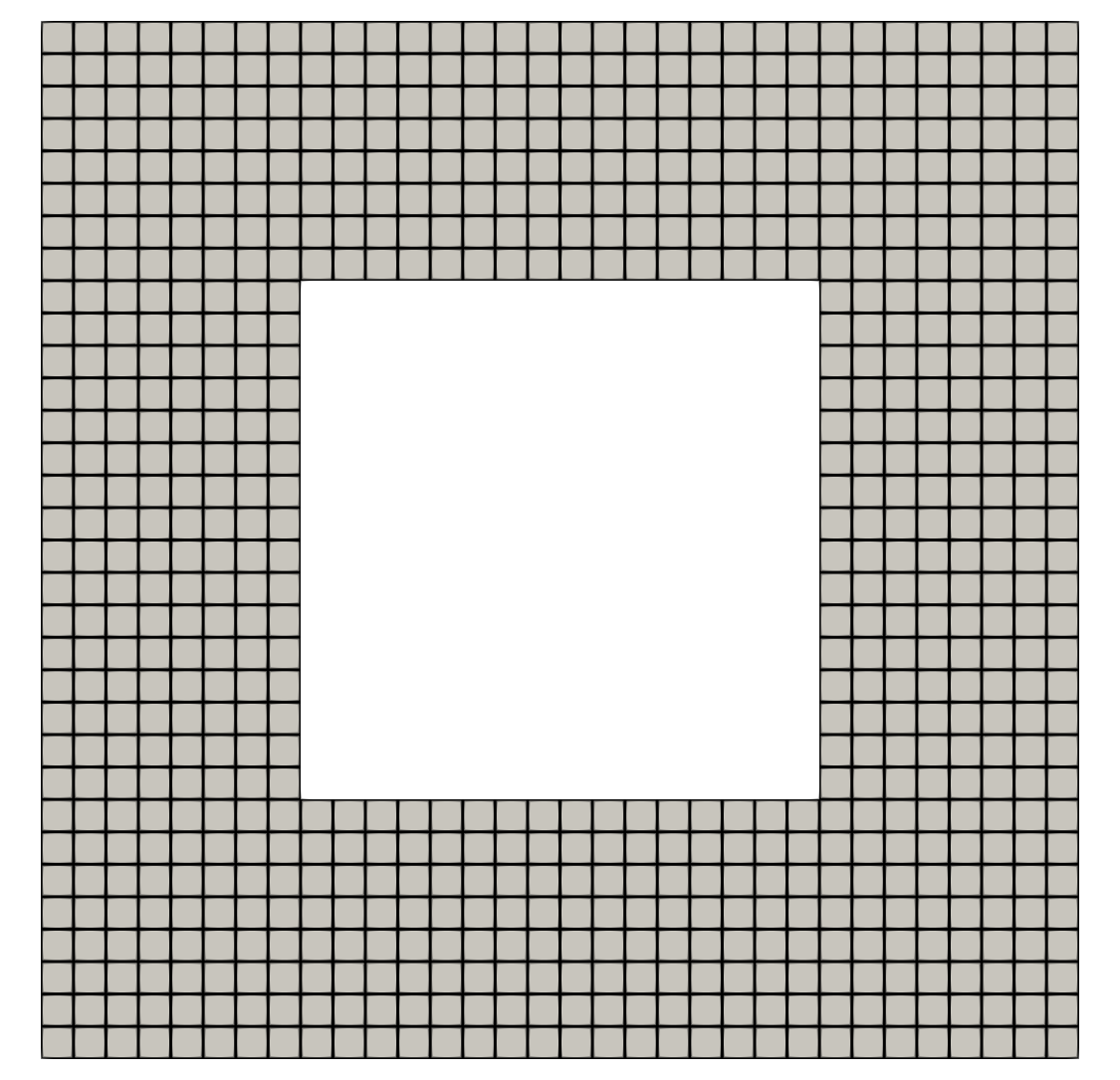}
\end{subfigure}%
\begin{subfigure}{0.33\textwidth}
\centering
\includegraphics[width=.99\linewidth]{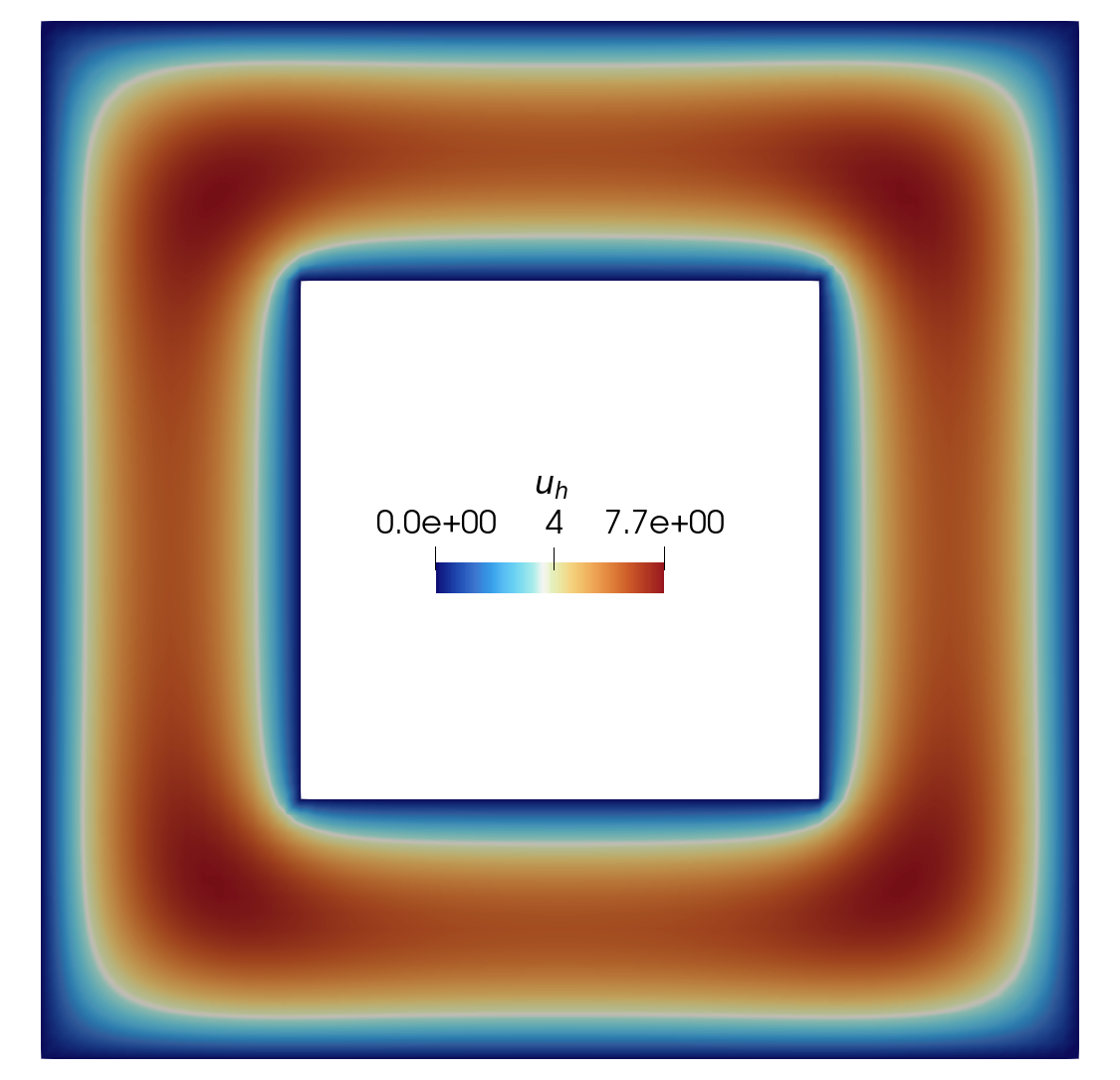}
\end{subfigure}%
\begin{subfigure}{0.33\textwidth}
\centering
\includegraphics[width=.99\linewidth]{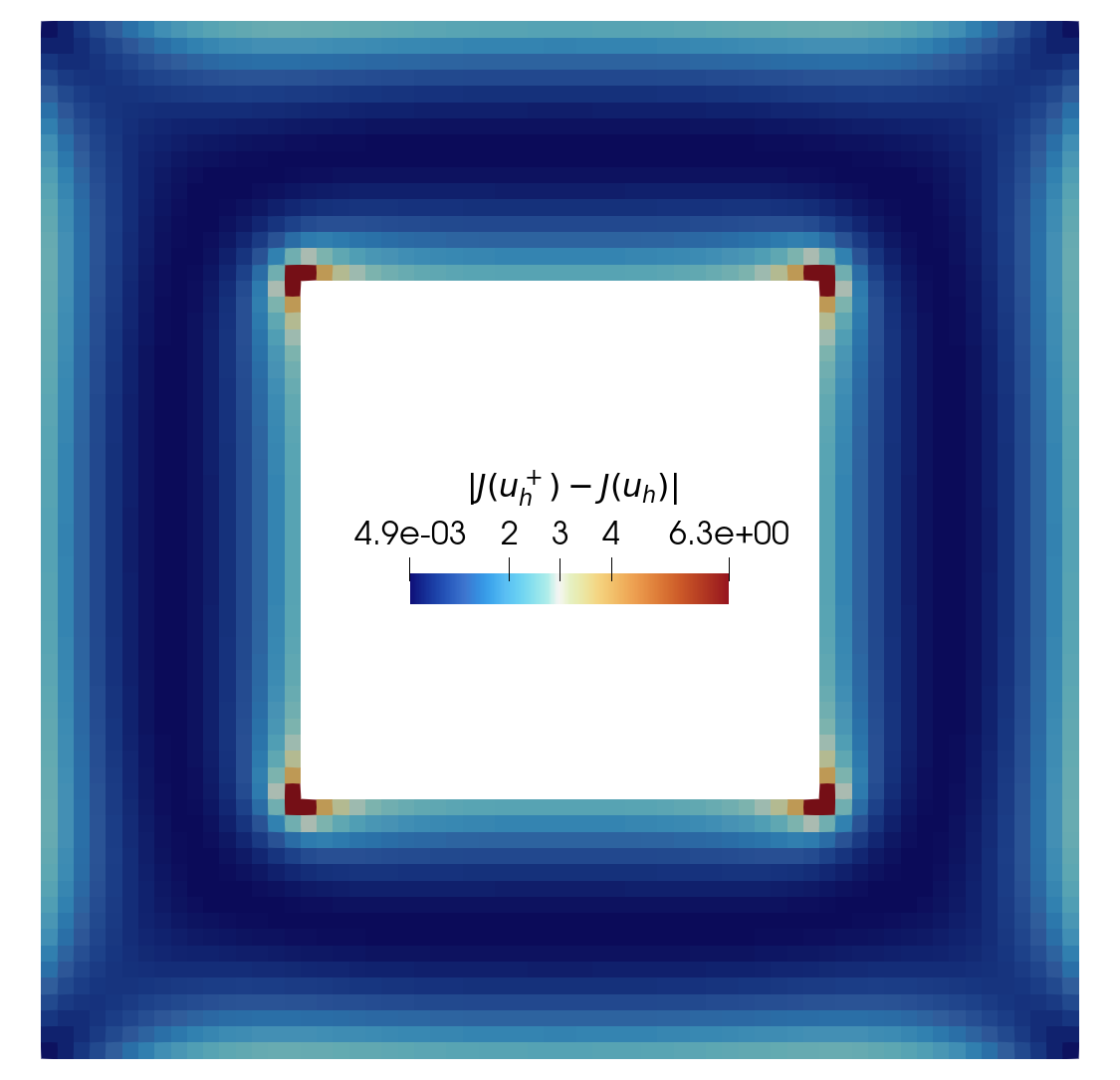}
\end{subfigure}
\caption{For the numerical demonstration:
(left) the domain $\Omega$ and computational mesh with size $h=\nicefrac{1}{16}$,
(center) the finite element solution $u_h$,
(left) a graphical approximation to element-level contributions to the error
$J(u) - J(u_h)$ by plotting $|J(\uplus) - J(u_h)| \bigr|_{\Omega^e}$, where
$\Omega^e$ corresponds to an individual element.}
\label{fig:results2}
\end{figure}

\end{results}

%%%%%%%%%%%%%%%%%%%%%%%%%%%%%%%%%%%%%%%%%%%%%%%%%%%%%%%%%%%%%%%%%%%%%%%%%%%%%%%
\section{Discussion}
\label{sec:discussion}
%%%%%%%%%%%%%%%%%%%%%%%%%%%%%%%%%%%%%%%%%%%%%%%%%%%%%%%%%%%%%%%%%%%%%%%%%%%%%%%

In this note, we have illustrated that a certain class of nonlinear
functionals can lead to unreliable error estimates \emph{for a particular
choice of adjoint problem statement, adjoint discretization, and
error estimate}. To elaborate, defining the adjoint problem itself presents
a choice about where linearizations occur. Presently, we have chosen the
point $u_h$ to carry about linearizations. Our motivation for doing so
stems primarily from familiarity with the \emph{discrete} goal-oriented
error estimation literature \cite{fidkowski2011review}, where this is common.
This background ultimately led us to investigations of linearization errors
\cite{granzow2023linearization}, which we have further understood in this
note.

It is also common, however, in the literature
\cite{becker2001optimal, endtmayer2024posteriori},
to carry about linearizations with respect to the exact solution $u$
and derive an error representation that is higher-order than the ones
considered in this note. This representation contains the primal
\emph{and} adjoint residuals as functions of the discretization errors $e:=u-u_h$ and
$e^*:=z-z_h$. In this context, it is common to approximate the
discretization error as $e \approx \uplus-u_h$ and, since one has
access to the fine-space primal solution, $\uplus$, it is possible to
approximate the adjoint problem by linearizing about the fine-space
solution $\uplus$. We suggest investigating linearization errors in
this setting as an avenue for future work.

Lastly, we remark that we have observed something similarly akin to
a loss of reliability in the context of \emph{nonlinear} variational problems
\cite{granzow2023linearization} and not simply within the setting of bilinear forms.
Further, we numerically observed poor error estimates for \emph{local}
quantities of interest (e.g. ones integrated over only a sub-domain), where
Galerkin orthogonality does not strictly hold. Although Theorems
\ref{thm:result1} and \ref{thm:result2} do not strictly assume symmetry
of the bilinear form, it is not clear to us that functionals can be
constructed so that $J'(u_h;e) = B(e;\varphi_h)$ for non-symmetric
bilinear forms. It would be interesting to know whether the present
analysis could be extended to such a context.
We leave investigation into both of these matters as an avenue for
interesting future work.

%%%%%%%%%%%%%%%%%%%%%%%%%%%%%%%%%%%%%%%%%%%%%%%%%%%%%%%%%%%%%%%%%%%%%%%%%%%%%%%


\begin{thebibliography}{10}

\bibitem{ainsworth1997posteriori}
Mark Ainsworth and J.~Tinsley Oden.
\newblock A posteriori error estimation in finite element analysis.
\newblock {\em Computer Methods in Applied Mechanics and Engineering},
  142(1-2):1--88, 1997.

\bibitem{becker1996feed}
Roland Becker and Rolf Rannacher.
\newblock {\em A feed-back approach to error control in finite element methods:
  Basic analysis and examples}, volume~4.
\newblock IWR, 1996.

\bibitem{becker2001optimal}
Roland Becker and Rolf Rannacher.
\newblock An optimal control approach to a posteriori error estimation in
  finite element methods.
\newblock {\em Acta numerica}, 10:1--102, 2001.

\bibitem{bornemann1996posteriori}
Folkmar~A Bornemann, Bodo Erdmann, and Ralf Kornhuber.
\newblock A posteriori error estimates for elliptic problems in two and three
  space dimensions.
\newblock {\em SIAM journal on numerical analysis}, 33(3):1188--1204, 1996.

\bibitem{carstensen2016justification}
Carsten Carstensen, Dietmar Gallistl, and Joscha Gedicke.
\newblock Justification of the saturation assumption.
\newblock {\em Numerische Mathematik}, 134(1):1--25, 2016.

\bibitem{endtmayer2020two}
B.~Endtmayer, U.~Langer, and T.~Wick.
\newblock Two-side a posteriori error estimates for the dual-weighted residual
  method.
\newblock {\em SIAM Journal on Scientific Computing}, 42(1):A371--A394, 2020.

\bibitem{endtmayer2024posteriori}
Bernhard Endtmayer, Ulrich Langer, Thomas Richter, Andreas Schafelner, and
  Thomas Wick.
\newblock A posteriori single-and multi-goal error control and adaptivity for
  partial differential equations.
\newblock {\em arXiv preprint arXiv:2404.01738}, 2024.

\bibitem{fidkowski2011review}
Krzysztof~J Fidkowski and David~L Darmofal.
\newblock Review of output-based error estimation and mesh adaptation in
  computational fluid dynamics.
\newblock {\em AIAA journal}, 49(4):673--694, 2011.

\bibitem{granzow2023linearization}
Brian~N Granzow, D~Thomas Seidl, and Stephen~D Bond.
\newblock Linearization errors in discrete goal-oriented error estimation.
\newblock {\em Computer Methods in Applied Mechanics and Engineering},
  416:116364, 2023.

\bibitem{granzow2017output}
Brian~N Granzow, Mark~S Shephard, and Assad~A Oberai.
\newblock Output-based error estimation and mesh adaptation for variational
  multiscale methods.
\newblock {\em Computer Methods in Applied Mechanics and Engineering},
  322:441--459, 2017.

\bibitem{oden2001goal}
J~Tinsley Oden and Serge Prudhomme.
\newblock Goal-oriented error estimation and adaptivity for the finite element
  method.
\newblock {\em Computers \& Mathematics with Applications}, 41(5-6):735--756,
  2001.

\bibitem{oden2002estimation}
J~Tinsley Oden and Serge Prudhomme.
\newblock Estimation of modeling error in computational mechanics.
\newblock {\em Journal of Computational Physics}, 182(2):496--515, 2002.

\bibitem{stewart1998tutorial}
James~R Stewart and Thomas~JR Hughes.
\newblock A tutorial in elementary finite element error analysis: A systematic
  presentation of a priori and a posteriori error estimates.
\newblock {\em Computer Methods in Applied Mechanics and Engineering},
  158(1-2):1--22, 1998.

\bibitem{verfurth1996review}
R{\"u}diger Verf{\"u}rth.
\newblock {\em A review of a posteriori error estimation and adaptive
  mesh-refinement techniques}.
\newblock Wiley-Teubner, 1996.

\end{thebibliography}
\end{document}